\documentclass[10pt]{article}
\usepackage{amssymb}
\usepackage{amsmath}
\usepackage{amsthm}
\usepackage{amsfonts}
\usepackage[ansinew]{inputenc}
\usepackage{amsthm,amsfonts}
\makeatletter \@addtoreset{equation}{section} \makeatother

\newtheorem{theorem}{Theorem}[section]

\newtheorem{lemma}{Lemma}[section]

\newtheorem{remark}{Remark}[section]

\makeatletter
\renewcommand*{\@biblabel}[1]{\hfill#1.}
\makeatother

\title{An Existence Result for the Generalized Vector Equilibrium Problem on Hadamard Manifold.}
\author{Batista E. E. A.\thanks{IME, Universidade Federal de Goi\'as,
Goi\^ania, GO 74001-970, BR ({\tt edvaldoelias1@hotmail.com}).}
\and G. C. Bento\thanks{IME, Universidade Federal de Goi\'as,
Goi\^ania, GO 74001-970, BR ({\tt glaydston@ufg.br}).}
\and
 Ferreira, O. P.
\thanks{IME, Universidade Federal de Goi\'as,
Goi\^ania, GO 74001-970, BR({\tt orizon@ufg.br}).}
}
\begin{document}
\maketitle

\vspace{-.3cm}

\begin{abstract}
IA sufficient condition for the existence of a solution for generalized vector equilibrium problem (GVEP) on Hadamard manifold, by using  a version of  KKM lemma  on this context,  is presented in this paper.  It is worth to point out that, in particular,   existence result  of solution for optimization problems, vector optimization problems, Nash equilibria problems, complementarity problems and variational inequality problems can be obtained as a special case of the existence result  for GVEP in this new context.
\end{abstract}


\section{Introduction} 
The generalized vector equilibrium problem (denoted by GVEP)  has been widely studied and is a very active field of research. One of the motivations is that several problems can be formulated as a generalized vector equilibrium problem, for instance, \mbox{optimization} problems, vector optimization problem, Nash equilibria problems, complementarity problems, fixed point problems and variational inequality problems. An extensive development is found, e.g., in    Fu~\cite{FU2000}, Fu and Wan \cite{FUAn2002}, Konnov and Yao\cite{KonnovYao1999}, Ansari et al. \cite{AnsarKonnovYao2001}, Farajzadeh et al. \cite{FarajzadehAmini-Harandi2008} and  their references. An important issue is under what conditions there exists a solution to GVEP. In the linear setting, several authors have provided results answering this question; see, for instance, , e.g., \cite{ANSARI1999},    Fu~\cite{FU2000}, Fu and Wan \cite{FUAn2002}, Konnov and Yao\cite{KonnovYao1999}, Ansari et al. \cite{AnsarKonnovYao2001}, Farajzadeh et al. \cite{FarajzadehAmini-Harandi2008} and  their references.

The first papers  dealing with the  subject on  existence of solution for equilibrium problems in the Riemannian context were Colao et al.~\cite{CLMM2012} and  Zhou and Huang\cite{ZhouHuang2009}, by generalizing  KKM lemma  to  Hadamard manifold.   By using   KKM lemma  on Hadamard manifold Zhou and Huang \cite{Zhou2013} have obtained  result of existence for vector optimization problems and vector variational inequality on this context.   In a similar manner   Li and Huang \cite{XiaHuang2013} have present result of existence for special class of  GVEP.  In the present work, we  use   KKM lemma  on Hadamard manifold  for proving existence result for GVEP. As far as we know it is  a new contribution of this paper.   It is worth to point out that our result include the results of  \cite{CLMM2012, Zhou2013} and not is included in the paper \cite{XiaHuang2013}.

The organization of the paper is as follows.
In Section \ref{sec2},  some notations and  basic results  used in the paper are presented. In Section \ref{s.main} the main results are stated and proved.  Some final remarks are made in Section~\ref{fr}.

\section{Notations and basics results}\label{sec2}
\subsection{Riemanian Geometry}
In this section, we recall some fundamental and basic concepts needed for reading this paper. These results and concepts can be found in the books on Riemannian geometry, see do Carmo~\cite{MP1992} and Sakay~\cite{S1996}.

Let $M$ be a $n$-dimensional connected manifold. We denote by $T_xM$ the $n$-dimensional {\it tangent space} of $M$ at $x$, by $TM=\cup_{x\in M}T_xM$ {\itshape{the tangent bundle}} of $M$ and by ${\cal X}(M)$ the space of smooth vector fields over $M$. When $M$ is endowed with a Riemannian metric $\langle .\,,\,. \rangle$, with the corresponding norm denoted by $\| . \|$, then $M$ is a Riemannian manifold. Recall that the metric can be used to define the length of piecewise smooth curves $\gamma:[a,b]\rightarrow M$ joining $x$ to $y$, i.e., $\gamma(a)=x$ and $\gamma(b)=y$, by $l(\gamma):=\int_a^b\|\gamma^{\prime}(t)\|dt$, and moreover, by minimizing this length functional over the set of all such curves, we obtain a Riemannian distance $d(x,y)$ inducing the original topology on $M$. We denote by $B(x,\epsilon)$ the Riemannian ball on $M$ with center $x$ and radius $\epsilon>0$. A vector field $V$ along $\gamma$ is said to be {\it parallel} iff $\nabla_{\gamma^{\prime}} V=0$. If $\gamma^{\prime}$ itself is parallel we say that $\gamma$ is a {\it geodesic}. Given that the geodesic equation $\nabla_{\ \gamma^{\prime}} \gamma^{\prime}=0$ is a second-order nonlinear ordinary differential equation, we conclude that the geodesic \mbox{$\gamma=\gamma _{v}(.,x)$} is determined by its position $x$ and velocity $v$ at $x$. It is easy to verify that $\|\gamma ^{\prime}\|$ is constant.  The restriction of a geodesic to a closed bounded interval is called a {\it geodesic segment}.  We usually do not distinguish between a geodesic and its geodesic segment, as no confusion can arise.  A Riemannian manifold is {\it complete} iff the geodesics are defined for any values of $t$. The \mbox{Hopf-Rinow's} Theorem~(\cite[Theorem 2.8, page 146]{MP1992} or~\cite[Theorem~{1.1}, page 84]{S1996}) asserts that, if this is the case, then  $( M, d)$ is a complete metric space and, bounded and closed subsets are compact. From the completeness of the Riemannian manifold $M$, the {\it exponential map} $\exp_{x}:T_{x}  M \to M $ is defined by $\exp_{x}v\,=\, \gamma _{v}(1,x)$, for each $x\in M$.  A complete simply-connected Riemannian manifold of nonpositive sectional curvature is called an Hadamard manifold.  It is known that if $M$ is a Hadamard manifold, then $M$ has the same topology and differential structure as the Euclidean space $\mathbb{R}^n$; see, for instance,~\cite[Lemma 3.2, page 149]{MP1992} or \cite[Theorem 4.1, page 221]{S1996}. Furthermore, are known some similar geometrical properties to the existing in Euclidean space $\mathbb{R}^n$, such as, given two points there exists an unique geodesic segment that joins them.  {\it In this paper, all manifolds $M$ are assumed to be Hadamard and finite dimensional}.

\subsection{Convexity} \label{sec3}


A set $\Omega\subset M$ is said to be {\it convex}  iff any
geodesic segment with end points in $\Omega$ is contained in
$\Omega$, that is, iff $\gamma:[ a,b ]\to M$ is a geodesic such that $x =\gamma(a)\in \Omega$ and $y =\gamma(b)\in \Omega$, then $\gamma((1-t)a + tb)\in \Omega$ for all $t\in [0,1]$.  For an arbitrary set $\mathcal{B} \subset M$ the {\it convex hull} of $\mathcal{B}$ denoted by $ \mbox{co}(\mathcal{B})$, that is, the smallest convex subset of $M$ containing $\mathcal{B}$. Let $\Omega\subset M$ be a convex set.  

\subsection{Set-valued mapping} \label{sec3}

For any  set $\mathcal{A}$ we denote $2^{\mathcal A}$  the set of all subset of $\mathcal{A}$.   Let $M$ be a Hadamard manifold,   $\Omega \subset M$ a nonempty  set. For $T: \Omega \rightarrow 2^{\mathbb{Y}}$ a set-valued mapping  the {\it domain} and the {\it range}  are the sets, respectively, defined by
\begin{equation} \label{eq:dr}
\mbox{dom}T:=\left\{ x\in \Omega ~: ~T(x)\neq \varnothing \right\}, \qquad   \mbox{rge} T:=\left\{ y \in {\mathbb{Y}} ~: ~y\in T(x) \mbox{ for some } x \in \Omega \right\}, 
\end{equation}
and the {\it inverse} is set-valued mapping $T^{-1}:= {\mathbb{Y}} \rightarrow 2^{\Omega} $ defined by
\begin{equation} \label{eq:dr}
T^{-1}(y):=\left\{x\in \Omega ~:~ y\in T(x) \right\}.
\end{equation}
A set-valued mapping $T: \Omega \rightarrow 2^{\mathbb{Y}}$ is said to be \emph{upper semicontinuous} on $\Omega$ if, for each $x_0\in \Omega$ and any open set $V$ in $\Omega$ containing $T(x_0)$, exists an open neighborhood $U$ of $x_0$ in $\Omega$ such that $T(x)\subset V$ for all $x\in U$. 

Next result is a version  of   KKM lemma  in Riemannian context  due to \cite{CLMM2012},  which is an extension to Hadamard manifold of KKM theorem see,  for example,  \cite{TARAF1987}.  
 \begin{lemma}  \label{le:colao}
Let $M$ be a Hadamard manifold,   $\Omega \subset M$ a nonempty closed convex set and  $G:\Omega \rightarrow 2^{\Omega}$  a set-valued mapping such that, for each $y\in\Omega, G(y)$ is closed. Suppose that
\begin{itemize}
\item [(i)] there exists $y_0\in\Omega$ such that $G(y_0)$ is compact;
\item [(ii)] $\forall ~ y_1,...,y_m\in \Omega,~\emph{co}(\{y_1,...,y_m\})\subset\bigcup_{i=1}^{m}G(y_i)$.
\end{itemize}
Then $\bigcap_{y\in \Omega}^{}G(y)\neq \varnothing.$
\end{lemma}
\begin{proof}
See, \cite[Lemma 3.1]{CLMM2012}. 
\end{proof}

\section{Generalized  vectorial   equilibrium problem} \label{s.main}

In this section, we present a sufficient condition for the existence of solution of generalized vector equilibrium problem on Hadamard manifolds. From now on, $\Omega \subset M$ will denote a nonempty closed convex set and $ \mathbb{Y}$ denotes  a topological vector space. Let $C: \Omega \rightarrow 2^{\mathbb{Y}}$ be set-valued mapping such that
\begin{equation} \label{eq:cx}
 C(x) ~  \mbox{ is a closed convex cone},  \qquad \mbox{int}\,C(x)\neq \varnothing, \qquad  \forall~   x \in \Omega.
\end{equation}
A set-valued mapping  $F:\Omega\times\Omega \rightarrow 2^{\mathbb{Y}}$  is called $C_x$ - \emph{quasiconvex-like} if for any geodesic segment $\gamma:[0, 1]\to \Omega$  there holds
 $$
 F(x,\gamma(t))\subseteq F(x,\gamma(0))-C(x) \quad ~ \mbox{or} ~ \quad  F(x,\gamma(t))\subseteq F(x,\gamma(1))-C(x), \qquad \forall \; t\in [0,1]. 
 $$
Then, given a set-valued mapping  $F: \Omega\times \Omega \rightarrow 2^{\mathbb{Y}}$, the {\it generalized vector equilibrium problem} in the Riemannian context (denoted by GVEP) consists in:
\begin{equation}\label{eq:p}
\mbox{Find $x^{*}\in \Omega$}:\quad F(x^\ast,y) \nsubseteq -\mbox{int}\,C(x^{*}),\qquad \forall \ y \in \Omega.
\end{equation}
Next result is closed related to   \cite[Theorem 2.1]{ANSARI1999}.  It  establish an existence result of solution for GVEP as an application  of Lemma~\ref{le:colao}. 

\begin{theorem} \label{th:main}
Let $F:\Omega\times\Omega \rightarrow 2^{\mathbb{Y}}$ be a set-valued mapping such that  for each $x, y\in\Omega$ 
\begin{itemize}
\item [{\bf h1}.]   $F(x,x)\not\subset -\emph{int}\,C(x)$;
\item [{\bf h2}.]  $F(\cdot,y)$ is upper semicontinuos;
\item [{\bf h3}.] $F$ is $C_x$-quasiconvex-like;
\item [{\bf h4}.] there exist $D\subset\Omega$   compact  and  $ y_0 \in\Omega$  such that $x\in\Omega\backslash D \Rightarrow F(x,y_0)\subset -\emph{int}\,C(x)$.
\end{itemize}
Then the solution set $S^{\ast}$ of  the GVEP defined in \eqref{eq:p} is a compact  nonempty set. 
\end{theorem}
From now on we assume that every assumptions on Theorem~\ref{th:main} hold.  In order to prove Theorem~\ref{th:main} we need of some  preliminaries. First we define the set-valued mapping  $P:\Omega \rightarrow 2^{\Omega}$ given by 
\begin{equation} \label{eq:set}
P(x):= \left\{y\in\Omega~:~F(x,y)\subset-\mbox{int}\,C(x) \right\}.
\end{equation}
\begin{lemma} \label{l:at1}
If    $S^{\ast}=\varnothing$ then, for each $x, y\in\Omega$, the set-valued mapping $P$  satisfies the following conditions:
\begin{itemize}
\item[i)]  the set $P(x)$ is nonempty and convex; 
\item[ii)] $P^{-1}(y)$ is an open set and  $\bigcup_{y\in\Omega}^{}P^{-1}(y)=\Omega$;
\item[iii)]  there exists $y_0\in\Omega$ such that $P^{-1}(y_0)^{c}$ is compact.
\end{itemize}
\end{lemma}
\begin{proof}
Since   the solution set $S^{\ast}=\varnothing$,  the definition  in \eqref{eq:set} allow us to conclude that $P(x)\neq\varnothing$, for all $x\in \Omega$, which proof the first statement in item $i$.  Let $x\in \Omega$. For proving that  $P(x)$ is convex, take    $y_1,y_2\in P(x)$ and a geodesic $\gamma: [0, 1] \to \Omega$ such that $\gamma(0)=y_1$ and $\gamma(1)=y_2$.  Using the assumption {\bf h1}  we have 
\begin{equation} \label{eq:itws}
 F(x,\gamma(t))\subseteq F(x,y_1)-C(x) \quad ~ \mbox{or} ~ \quad  F(x,\gamma(t))\subseteq F(x,y_2)-C(x). 
\end{equation}
As $y_1,y_2\in P(x)$, definition of $P(x)$ in \eqref{eq:set} implies that $F(x,y_1)\subset-\mbox{int}\,C(x)$  and $F(x,y_2)\subset-\mbox{int}\,C(x)$. Hence,  taking in account that $-\mbox{int}\,C(x) - C(x) \subset -\mbox{int}\,C(x)$  it follows from \eqref{eq:itws} that
$$
 F(x,\gamma(t))\subset -\mbox{int}\,C(x),
$$
and the proof of item $i$ is concluded.

Now we are going to prove item $ii$.  First of all note that using definition in  \eqref{eq:dr} we have
\begin{equation} \label{eq:pm1}
P^{-1}(y)=\{x\in \Omega~: ~y\in P(x)\}=\{x\in\Omega ~: ~F(x,y)\subset-\mbox{int}\, C(x)\}, 
\end{equation}
where the second equality follows from definition of the set $P(x)$ in \eqref{eq:set}.   Take  $x_0\in P^{-1}(y)$. From the second equality in \eqref{eq:pm1} and   that  $-\mbox{int}\,C(x)$ is an open set, using  {\bf h2},  there exists  $ V_{x_0}\subset \Omega$   an open set  such that  
$
 F(x,y)\subset-\mbox{int}\,C(x), 
$
for all $ x\in V_{x_0}.$
Therefore, $P^{-1}(y)$ is open, which proof the first part in  item $ii$.  Definition in \eqref{eq:pm1} implies that $P^{-1}(y)\subseteq \Omega$ for all $y\in \Omega$. For concluding the proof of item $ii$, is sufficient to prove that  $\Omega\subseteq \bigcup_{y\in\Omega}^{}P^{-1}(y)$.   Let  $x\in\Omega$. From  Item $i$  we  have $P(x)\neq\varnothing$ and hence  there exists $y\in P(x)$. Therefore  $x\in P^{-1}(y)$ for some $y\in \Omega$, which conclude the proof  of item $ii$. 

For proving  item $iii$, first note that  from assumption {\bf h4} and definition in \eqref{eq:pm1},  we have
$$
P^{-1}(y_0)^{c}=\{x\in\Omega~:~F(x,y_0)\not\subset-\mbox{int}\,C(x)\}\subset D, 
$$
for some $y_0\in \Omega$ and $D\subset \Omega$ a  compact set.  From item $i$  the set $P^{-1}(y_0)$ is an open set,  and since  $D$ is a compact set,  we conclude from last inclusion  that $P^{-1}(y_0)^{c}$  is a compact set and the proof of the proposition is concluded.
\end{proof}
Now we are ready to prove our main result, namely,   Theorem~\ref{th:main}.
\begin{proof} 
Let us suppose, by contradiction,  that he solution set   $S^{\ast}=\varnothing$. Let $ G :   \Omega \to  2^{\Omega}$ be  the set-valued mapping defined by
\begin{equation} \label{eq:faux}
G(y):=P^{-1}(y)^{c}.
\end{equation}
Define  the set $D:=\bigcap_{y\in \Omega}^{}G(y)$. We have two passibility  for the set $D$, namely,   $D\neq \varnothing$ or   $D= \varnothing$. If  $D\neq \varnothing$, i.e.,  $\bigcap_{y\in \Omega}^{}P^{-1}(y)^{c}\neq \varnothing$,  then we have 
$
\bigcup_{y\in \Omega}^{}P^{-1}(y)\neq \Omega,
$
which is a contradition with the item $ii$ in  Lemma~\ref{l:at1}. Hence,  we conclude that  $D= \varnothing$, i. e., 
$$
\bigcap_{y\in \Omega}^{}G(y)= \varnothing .
$$
Thus, as we are under the assumption  $S^{\ast}=\varnothing$, combining definition in \eqref{eq:faux} and    Lemma~\ref{l:at1} items ii and iii, we conclude  that,  for each $y\in\Omega$,  the set $G(y)$ is closed and  there exists $y_0\in\Omega$ such that $G(y_0)$ is a compact set. Hence,  as $\bigcap_{y\in \Omega}^{}G(y)= \varnothing$,  Lemma~\ref{le:colao}  implies  that there exist  $ y_1,...,y_m\in \Omega$ such that  
$$
\mbox{co}\{y_1,...,y_m\}\not\subset \bigcup_{i=1}^{m}G(y_i), 
$$
and  thus  there exists $x\in \mbox{co}\{y_1,...,y_m\}$ such that  $x\notin G(y_i)=P^{-1}(y_i)^{c} $ for all $i=1, \ldots m$ or,  equivalently,    there exists $x\in \mbox{co}\{y_1,...,y_m\}$ such that  $x\in P^{-1}(y_i) $ for all $i=1, \ldots m$.  Therefore,  we conclude that
\begin{equation} \label{eq:rmth}
\exists ~ y_1,...,y_m\in \Omega,  \quad \exists ~ x\in \mbox{co}\{y_1,...,y_m\};  \quad   y_i\in P(x), \quad    \forall ~i=1, \ldots m.
\end{equation}
   Taking into  account that  $S^{\ast}=\varnothing$, items i of Lemma~\ref{l:at1}  implies that $P(x)$ is convex, which together with  relations in \eqref{eq:rmth} yields    
   $$
   \exists ~  x  \in \Omega:~ ~x\in P(x).
   $$
 Last inclusion and definition in \eqref{eq:set} imply   that there exists  $x  \in \Omega$ such that  $F(x,x)\subset-\mbox{int}\,C(x)$, obtaining  a contradiction with the assumption {\bf h1} in Theorem~\ref{th:main}. Therefore,  the solution set  $S^{\ast}\neq\varnothing$ and the proof of Theorem~\ref{th:main} is concluded. 
\end{proof}

\begin{remark}
Note that, in particular,  when $M=\mathbb{R}^{n}$, our problem~\eqref{eq:p} retrieves a particular instance of the generalized vector equilibrium problem studied by Ansari and Yao in \cite{ANSARI1999}. In the case where $C(x)=\mathbb{R}_{+}$, for each $x\in \Omega$, $\mathbb{Y}=\mathbb{R}$ and $F$ is single-valued map from $\Omega \times \Omega$ to $\mathbb{R}$, then the problem \eqref{eq:p} reduces to the equilibrium problem on Hadamard manifold considered by Colao et al. in \cite{CLMM2012}; see also Bento et al.~\cite{BCSS2013}. Now, let us consider the following vector optimization problem on Hadamard manifold:
\begin{equation} \label{eq:vopt}
\left\{\begin{array}{c}
w-\emph{min}f(x),\\
s. t.\; x\in \Omega,
\end{array}\right.
\end{equation}
where $f:M\to\mathbb{R}^{m}$ vector function and $w-\emph{min}$ denotes weak minimum.  Zhou and Huang~\cite{Zhou2013} presented an existence result of solutions for the problem \eqref{eq:vopt} by showing the equivalence of this with the variational inequality problem on Hadamard manifold (studied by N\'emeth in \cite{N2003}):
\begin{equation}\label{prob:VVI}
\mbox{Find $x^{*}\in \Omega$}:\quad \langle A(x^*), \exp^{-1}_{x^*}y\rangle \notin -\mathbb{R}^m_{++},\qquad \forall \ y \in \Omega,
\end{equation}
in the particular case where $f$ is a differentiable and convex vector function  and $A$ is the Riemannian Jacobian of  $f$. Taking into account that $x^{*}\in \Omega$ is a weak minimum of \eqref{eq:vopt} iff
\[
f(x)-f(x^{*})\notin -\mathbb{R}_{++}^{m}, \qquad \forall x\in \Omega,
\]
we point out that, in particular,   Theorem~\ref{th:main} is an existence result of solution for the problem \eqref{eq:vopt}  even in the case where f is only quasi convex and not necessarily differentiable.
\end{remark}
\section{Final remarks} \label{fr}
In this paper we study  basics intrinsic properties of  generalized vector equilibrium problem in Hadamard manifolds, and we
touch only slightly  the equilibrium problem theory in this  context.  We expect that the results of this paper become a 
first step towards a more general theory, including hyperbolic spaces and algorithms for solving that problems on Hadamard manifolds. We foresee further progress in this topic in the  nearby future.

\end{document}